\documentclass{amsart} 

\RequirePackage{lmodern}
\RequirePackage{t1enc}
\swapnumbers
\newtheorem{theorem}[equation]{Theorem}

\newtheorem{proposition}[equation]{Proposition}
\newtheorem{corollary}[equation]{Corollary}

\theoremstyle{remark}

\newtheorem{remark}[equation]{Remark}
\theoremstyle{definition}

\numberwithin{equation}{section}
\def\@secnumfont{\bfseries}

\newcommand{\re}{\mathop{\rm\ Re}}

\newcommand{\lbd}{\lambda}

\renewcommand{\phi}{\varphi}

\newcommand{\bld}[1]{\boldsymbol{#1}}

\newcommand{\N}{\bld{N}}

\newcommand{\ignore}[1]{{}}

\newcommand{\A}{\boldsymbol{A}}
\newcommand{\D}{\mathcal{D}}

\title[Semigroups of measures]{Convergence of semigroups of measures\\ on a Lie group}
\author{Pawe{\l} G{\l}owacki}

\begin{document}
\thispagestyle{empty}
\begin{abstract}
 A theorem of Siebert asserts that if $\mu_n(t)$ are semigroups of probability measures on a Lie group $G$, and $P_n$ are the corresponding generating functionals, then
\[
\langle\mu_n(t),f\rangle\ \xrightarrow[n]{}\ \langle\mu_0(t),f\rangle,
\qquad
f\in C_b(G),\ t>0, 
\]
implies
\[
 \langle\pi_{P_n}u,v\rangle\ \xrightarrow[n]{}\ \langle\pi_{P_0}u,v\rangle,
\qquad
u\in C^{\infty}(\pi), v\in X,
\]
for every unitary representation $\pi$ of $G$ on a Hilbert space $X$, where $C^{\infty}(\pi,X)$ denotes the space of smooth vectors for $\pi$.

The aim of this note is to give a simple proof of the theorem and propose some improvements. In particular, we completely avoid employing unitary representations by showing simply that under the same hypothesis  
\[
 \langle P_n,f\rangle\ \xrightarrow[n]{}\ \langle P_0,f\rangle,
\qquad
f\in C_b^2(G).
\]
As a corollary, the above thesis of Siebert is extended to strongly continuous  representations of $G$ on Banach spaces.
\end{abstract}
\maketitle
\section{Introduction}

Let $X$ be a Banach space and
\[
 \A_n:\D_n\to X,
\qquad
\D_n={\rm dom}(\A_n)\subset X,
\]
inftinitesimal generators of strongly continuous contraction semigroups $e^{t\A_n}$ on $X$. It is classical that 
\begin{equation}\label{se}
 e^{t\A_n}f\xrightarrow[n]{} e^{t\A_0}f,
\qquad
f\in X, \ t>0,
\end{equation}
is equivalent to
\begin{equation}\label{re}
 (\lbd I-\A_n)^{-1}f\xrightarrow[n]{}(\lbd I-\A_0)^{-1}f,
\qquad
f\in X, \ \re\lbd<0.
\end{equation}
(See, e.g. Yosida \cite{yosida}, IX.12.) Furthermore, if there exists a common core domain $\D\subset\D_n$, then
\begin{equation}\label{ge}
 \A_nf\xrightarrow[n]{} \A_0f,
\qquad
f\in\D,
\end{equation}
implies (\ref{se}). See Kato \cite{kato}, Theorem VIII.1.5. 

A remarkable property of semigroups of probability measures is that  this implication can be in a way reversed. Namely, if $\mu_n(t)$ is a sequence of such semigroups on a Lie group $G$, and $P_n$ are the corresponding generating functionals, then
\begin{equation}\label{meas}
\langle\mu_n(t),f\rangle\ \xrightarrow[n]{}\ \langle\mu_0(t),f\rangle,
\qquad
f\in C_b(G),\ t>0, 
\end{equation}
implies
\[
 \langle\pi_{P_n}u,v\rangle\ \xrightarrow[n]{}\ \langle\pi_{P_0}u,v\rangle,
\qquad
u\in C^{\infty}(\pi), v\in X,
\]
for every unitary representation $\pi$ of $G$ on a Hilbert space $X$, where $C^{\infty}(\pi,X)$ denotes the space of smooth vectors for $\pi$, see Siebert\cite{siebert}, Proposition 6.4. A generalisation can be found in Hazod \cite{hazod}.

The aim of this note is to give a simple proof of the theorem and propose some improvements.
Thus, it is  shown that (\ref{meas}) implies
\[
 \langle P_n,f\rangle\ \xrightarrow[n]{}\ \langle P_0,f\rangle,
\qquad
f\in C_b^2(G).
\] 
The main idea is that the norm of a dissipative distribution interpreted as a functional on $C_b^2(G)$ can be controlled by its action on the coordinate functions. This helps to eliminate any reference to unitary representatations so prominent in Siebert \cite{siebert}. The result of Siebert is recovered as a corollary and extended to  strongly continuous representations on Banach spaces such that
\begin{equation}\label{int}
 \int_{G\setminus U}\|\pi(x)\|P(dx)<\infty,
\end{equation}
for an open neighbourhood $U$ of the identity. 
 
The striking simplicity of our proof as compared to that of Siebert \cite{siebert} is our main argument for the presentation. However, an essential step in the proof is based on an idea borrowed from Siebert which reduces the task to the case where the norms $\|P_n\|$ as functionals on $C_b^2(G)$ are uniformly bounded in $n$ (cf. Siebert \cite{siebert}, Proposition 6.3).

As an introduction to the theory of probability semigroups of measures on Lie groups we recommend Duflo \cite{duflo} and Hulanicki \cite{hulanicki}. The reader may also wish to consult Hazod-Siebert \cite{hazod-siebert}.

\section{Dissipative distributions}
Let $G$ be a Lie group and $\{X_j\}_{j=1}^d$ a basis of leftinvariant vector fields $G$. Let $C_b(G)$ denote the Banach space of bounded continuous functions with the sup norm $\|\cdot\|_{\infty}$. Let
$$
C^{2}_b(G)=\{f\in C^{2}(G): \|f\|_{C^2}<\infty\},
$$
where
$$
\|f\|_{C^2}=\max_{|\alpha|\le 2}\|X^{\alpha}f\|_{\infty}.
$$

A Schwartz distribution $P$ is said to be \textit{dissipative} if
\[
 \langle P,f\rangle\le0
\]
for real $f\in C_c^{\infty}(G)$ which attain maximal value at the identity $e$. Moreover, $P$ extends to a continuous linear functional on $C_b^2(G)$. If $f\in C^2_b(G)$ and 
\[
 f(e)=X_jf(e)=X_jX_kf(e)=0,
\qquad
1\le j,k\le d,
\]
then
\begin{equation}\label{levy}
\langle P,f\rangle=\int_{G}f(x)P(dx).
\end{equation}
 In particular, $P$ coincides with a positive Radon measure on $G\setminus\{e\}$ which is bounded outside any open neighbourhood of $e$.
It is convenient to introduce coordinate functions near the identity. Let $\Phi_j\in C_c^{\infty}(G)$ be real functions such that
\[
X_j\Phi_k(e)=\delta_{jk},
\qquad
1\le k,j\le d,
\]
and let
\[
 \Phi_{jk}=\Phi_j\Phi_k,
\qquad
1\le k,j\le d.
\]
Let also $\Phi\in C_c^{\infty}(G)$ be a $[0,1]$-valued function equal to $1-\sum_{k=1}^d\Phi_j^2$ in a neighbourhood of $e$. Then we have the following Taylor estimate
\begin{equation}\label{taylor}
\bigg|f(x)-f(e)-\sum_{j=1}^dX_jf(e)\Phi_j(x)\bigg|\le C\|f\|_{C_b^2(G)}(1-\Phi(x))
\end{equation}
for $f\in C_b^2(G)$. The constant $C$ here and throughout the paper is a generic constant which may vary from statement to statement. 

Since  $\langle P,1-\Phi\rangle\ge0$ and
\[
\langle P,\Phi\rangle+\langle P,1-\Phi\rangle=\langle P,1\rangle\le0,
\]
we have, by (\ref{levy}),
\begin{equation}\label{levy1}
\int1-\Phi(x)P(dx)=\langle P,1-\Phi\rangle\le|\langle P,\Phi\rangle|.
\end{equation}

\begin{remark}
The function $1-\Phi$ is often called a Hunt function. 
\end{remark}

\begin{proposition}\label{norm}
Let $P$ be dissipative. Then, there exists a constant $C>0$ such that
\[
 \|P\|\le C\bigg(|\langle P,\Phi\rangle|+\sum_{j,k=1}^d|\langle P,\Phi_{jk}\rangle|+\sum_{j=1}^d|\langle P,\Phi_j\rangle|+|\langle P,1\rangle|\bigg),
\]
where $\|P\|$ is the norm of $P$ as a linear functional on $C^2_b(G)$.
\end{proposition}
\begin{proof}
For every $f\in C_b^2(G)$ we can write 
\[
f(x)=f_1+f_2,
\qquad
f_1(x)=f(e)+\sum_{j=1}^dX_jf(e)\Phi_j(x)+\frac{1}{2}\sum_{j,k=1}^dX_jX_kf(e)\Phi_{jk},                        \]
and $f_2=f-f_1$. Then,
\[
 |\langle P,f_1\rangle|\le\bigg(\sum_{j,k=1}^d|\langle P,\Phi_{jk}\rangle
+\sum_{j=1}^d|\langle P,\Phi_j\rangle|+|\langle P,1\rangle|\bigg)\cdot\|f\|_{C_b^2(G)}
\]
and, by (\ref{levy}) and(\ref{taylor}), 
\begin{align*}
 |\langle P,f_2\rangle|&\le\left|\int f_2(x)P(dx)\right|\\
&\le\left|\int f_2(x)\Phi(x)P(dx)\right|
+\int|f_2(x)|\bigg(1-\Phi(x)\bigg)P(dx)\\
&\le C\int1-\Phi(x)P(dx)\cdot\|f\|_{C^2(G)}+\langle P,1-\Phi\rangle\cdot\|f\|_{C_b(G)},
\end{align*}
which, by (\ref{levy1}), completes the proof.
\end{proof}

\section{Convergence}

By the Hunt theory (cf. Hunt \cite{hunt} and Hulanicki \cite{hulanicki}), any dissipative $P$ is a generating functional for a weakly continuous semigroup $\mu(t)$ of subprobability measures. (The measures are probability measures if and only if $\langle P,1\rangle=0$.) Let $\pi$ be a strongly continuous representation of $G$ on a Banach space $X$ satisfying (\ref{int}). The operators 
\[
 \pi_{\mu(t)}u=\int \pi(x)\mu(t)(dx),
\qquad
u\in X,
\]
form a strongly continuous contraction semigroup, and
\begin{equation}\label{repr}
 \langle\pi_Pu,v\rangle=\langle P,f_{u,v}\rangle,
\qquad
u\in C^2(\pi,G),\ v\in X',
\end{equation}
is the infinitesimal generator of $\pi_{\mu(t)}$ with
\begin{equation}\label{coeff}
 C^2(\pi,X)=\{u\in X:f_{u,v}\in C^2_b(G)\},
\qquad
f_{u,v}(x)=\langle\pi(x)u,v\rangle,
\end{equation}
for its core domain. (Cf. Duflo \cite{duflo}, Section 11 and 12.)
In particular, the convolution operators
\begin{equation}\label{operators}
T(t)f(x)=f\star\widetilde{\mu}(t)=\int f(xy)\mu(t)(dy)
\end{equation}
form a  strongly continuous semigroup of contractions on  $C_b(G)$. The infinitesimal generator of $T(t)$ 
is the convolution operator
\[
\boldsymbol{P}f(x)=f\star\widetilde{P}(x)=\int f(xy)P(dy)
\]
for which $C_b^2(G)$ is a core domain.

\begin{remark}\label{unif}
Recall that if $T_n(t)$ are strongly continuous contraction semigroups on a Banach space $X$ and 
\[
T_n(t)u\xrightarrow[n]{} T_0(t)u,
\qquad
u\in X,\ t>0, 
\]
then for every fixed $u\in X$ the convergence is uniform in $0\le t\le1$. (See, e.g. Yosida \cite{yosida}, Theorem IX.12.1.)
\end{remark}

\begin{theorem}\label{goal}
Let $P_n$ be dissipative distributions on $G$. Denote by $\mu_n(t)$ the corresponding semigroups of  measures. If  
\[
 \langle\mu_n(t),f\rangle\ \xrightarrow[n]{}\ \langle\mu_0(t),f\rangle,
\qquad
f\in C_b(G),\ t>0,
\]
then, for every $f\in C_b^2(G)$,
\[
 \langle P_n,f\rangle\ \xrightarrow[n]{}\ \langle P_0,f\rangle.
\]
Moreover, if $P_0=0$, then $\|P_n\|\xrightarrow[n]{}0$.
\end{theorem}

\begin{proof}
Let us keep the notation established above. For $f\in C_b^2(G)$, we have
\begin{align*}
\langle&\mu_0(t)-\mu_n(t),f\rangle=\langle\mu_0(t)-\delta_0,f\rangle-\langle\mu_n(t)-\delta_0,f\rangle\\
&=\int_0^t\langle\mu_0(s)\star P_0,f\rangle\,ds-\int_0^t\langle\mu_n(s)\star P_n,f\rangle\,ds\\
&=\int_0^t\langle P_0,f\star\widetilde{\mu}_0(s)\rangle\,ds-\int_0^t\langle P_n,f\star\widetilde{\mu}_n(s)\rangle\,ds\\
&=\int_0^t\langle P_0,T_0(s)f-f\rangle\,ds-\int_0^t\langle P_nf,T_n(s)f-f\rangle\,ds
+t\langle P_0-P_n,f\rangle\\
&=\int_0^t\langle P_0-P_n,T_0(s)f-f\rangle\,ds\\
&+\int_0^t\langle P_n,T_0(s)f-T_n(s)f\rangle\,ds
+t\langle P_0-P_n,f\rangle,
\end{align*}
whence
\begin{align*}
\langle P_0-P_n,f\rangle&=\frac{1}{t}\langle\mu_0(t)-\mu_n(t),f\rangle
+\frac{1}{t}\int_0^t\langle P_n-P_0,T_0(s)f-f\rangle\,ds\\
&+\frac{1}{t}\int_0^t\langle P_n,T_n(s)f-T_0(s)f\rangle\,ds.
\end{align*}

Assume for the moment that 
\begin{equation}\label{bound}
\sup_n\|P_n\|<\infty. 
\end{equation}
Since, by Remark \ref{unif}, the convergence of the semigroups is uniform for $0\le s\le t\le1$, the last formula implies the desired convergence of generating functionals. In fact, pick $t>0$ small enough to make the second term small, then fix $t$ and take $n$ large enough to make the remaining terms small. If , in particular $P_0=0$, then, by Proposition \ref{norm}, 
\begin{align*} 
\|P_n\|\le& C\bigg(|\langle P_n,\Phi\rangle|+\sum_{j,k=1}^d|\langle P_n,\Phi_{jk}\rangle|\\
&+\sum_{j=1}^d|\langle P_n,\Phi_j\rangle|+|\langle P_n,1\rangle|\bigg)\xrightarrow[n]{}0, 
\end{align*}
 which proves the last statement of the theorem. Thus, it remains to show that (\ref{bound}) holds under the hypothesis of the theorem.

In fact, assume \textit{a contrario} that this is not true. Then there exists a sequence of integers $n_k$ such that $\alpha_k=\|P_{n_k}\|\to\infty$. The generating functionals  $Q_k=\alpha_k^{-1}P_{n_k}$ satisfy  
\begin{equation}\label{qs}
\|Q_k\|=1,
\qquad
k\in\N, 
\end{equation}
and the corresponding sequence of semigroups is
$$
\nu_k(t)=\mu_{n_k}(\alpha_k^{-1}t)=\bigg(\mu_{n_k}(\alpha_k^{-1}t)-\mu_{0}({\alpha_k^{-1}t})\bigg)+\mu_0(\alpha_k^{-1}t),
$$
which converges to the trivial semigroup, the generating functional being $Q_0=0$. The last statement follows from the fact that, by Remark \ref{unif}, 
\[
\langle\mu_{n_k}(t),f\rangle\ \xrightarrow[k]{}\ \langle\mu_0(t),f\rangle
\]
uniformly in $0\le t\le1$, and $\alpha_k\to\infty$. By the first part of the proof, $\|Q_k\|\to0$, which contradicts (\ref{qs}). This proves (\ref{bound}) and completes the proof of the theorem.
\end{proof}

\begin{corollary}
Let $\pi$ be a strongly continuous representation of $G$ on a Banach space $X$ satisfying (\ref{int}). If $\mu_n(t)$ and $P_n$ satisfy the hypothesis of Theorem \ref{goal}, then for every $u\in C^2(\pi,X)$ and every $v\in X'$,
\[
 \langle\pi_{P_n}u,v\rangle\ \xrightarrow[n]{}\ \langle\pi_{P_0}u,v\rangle.
\]
 \end{corollary}
\begin{proof}
If $u\in C^2(\pi,X)$, then, by (\ref{coeff}), $f_{u,v}$ is in $C_b^2(G)$ so, by Theorem \ref{goal} and (\ref{repr}),
\[
\langle\pi_{P_n}u,v\rangle=\langle P_n,f_{u,v}\rangle\ \xrightarrow[n]{}\ \langle P_0,f_{u,v}\rangle=\langle\pi_{P_0}u,v\rangle. 
\]
\end{proof}

\end{document}